\documentclass[12pt]{amsart}

\usepackage{amsmath}
\usepackage{latexsym}
\usepackage{amsfonts}
\usepackage{amssymb}
\newtheorem{thm}{Theorem}[section]
\newtheorem{lem}[thm]{Lemma}
\newtheorem{prop}[thm]{Proposition}

\theoremstyle{definition}
\newtheorem*{defi}{Definition}
\newtheorem*{obs}{Observation}

\newcommand{\R}{\mathbb R}
\newcommand{\C}{\mathbb C}

\newcommand{\N}{\mathbb N}

\begin{document}

\title{On approximations by shifts of the Gaussian function.}%

\author{Gerard Ascensi}
\thanks{The author is supported by MTM2005-08984-C02-01 and 2005SGR00611 projects and by the Fundaci\'{o} Cr\`{e}dit Andorr\`{a}.}
\address{Faculty of Mathematics, University of Vienna, Nordbergstrasse 15, 1090 Vienna, Austria}
\email{gerard.ascensi@univie.ac.at}


\begin{abstract}
The paper study the discrete sets of translations of the Gaussian function that span the spaces $L^1(\R)$ and $L^2(\R)$.
\end{abstract}

\maketitle

\section{Introduction.}

The study of the sets of translations of functions that span $L^p(\R)$ spaces is a classical topic in harmonic analysis. One wants to determine under what conditions a sequence $\{\varphi(t-\lambda)_{\lambda\in\Lambda}$ span these spaces for a function $\varphi\in L^p(\R)$ and a set $\Lambda\subset\R$. Wiener's Tauberian theorem \cite{wiener} asserts that $\{\varphi(t-\lambda)\}_{\lambda\in\R}$ span $L^2(\R)$ if and only if the set of points at which the Fourier transform of $\varphi$ vanishes has measure $0$, and span $L^1(\R)$ if this set is empty. It is natural to consider this problem when $\Lambda$ is a discrete set. We give an overview of the results on this topic in the last section.

This paper deals with the spanning properties of a discrete set of translations of the Gaussian function $\phi(t)=e^{-\pi t^2}$.  This problem was first studied by Zalik in \cite{Zal}. He considered a more general point of view (functions like the Gaussian), and he obtained necessary and sufficient conditions for the set $\Lambda$ that depend on
\begin{equation*}
S(\varepsilon)=\sum_{\lambda\in\Lambda}\frac{1}{|\lambda|^{2+\varepsilon}}.
\end{equation*}
He proved that a necessary condition for $\{\phi(t-\lambda)\}_{\lambda\in\Lambda}$ to span $L^2(\R)$ is the divergence of $S(0)$ and that a necessary condition is the divergence of $S(\varepsilon)$ for some $\varepsilon>0$.

The question is what happens with the sets such that $S(0)$ is divergent but $S(\varepsilon)$ is convergent for all positive $\varepsilon$. These kind of sets are known as sets of order $2$. The density of such a set is defined as:
\begin{equation*}
\Delta(\Lambda)=\lim_{r\to\infty}\frac{|\Lambda\cap(-r,r)|}{r^2}=\lim_{r\to\infty}\frac{n_{\Lambda}(r)}{r^2}.
\end{equation*}
In this expression $n_{\Lambda}$ is the counting function of $\Lambda$. If we write $\Lambda=\{\lambda_n\}_{n\in\N}$ with $|\lambda_n|\leqslant|\lambda_{n+1}|$, the density of $\Lambda$ coincides with $\lim_{n\to\infty}n/|\lambda_n|^2$. $S(0)$ is divergent and $S(\varepsilon)$ is convergent for those sets $\Lambda$ such that $0<\Delta(\Lambda)<\infty$. If $S(0)<\infty$ then the density will be $0$ and the divergence of $S(\varepsilon)$ implies infinite density. For sets without density (oscillation of the limit) we can define upper and lower versions, but we will not use this generalization in this note.

Our principal result is the next one:
\begin{thm}\label{teorema}
Let $\Lambda\subset\R$ be a discrete set and $\phi(t)=e^{-\pi t^2}$ the Gaussian function. Let $\Lambda^+=\Lambda\cap[0,\infty)$ and $\Lambda^-=\Lambda\cap(-\infty,0]$ the positive and negative parts of $\Lambda$.
\begin{enumerate}
\item If $\Delta(\Lambda^+)<1/2$ and $\Delta(\Lambda^-)< 1/2$ then $\{\phi(t-\lambda)\}_{\lambda\in\Lambda}$ does not span $L^2(\R)$.

\item If $\Delta(\Lambda^+)>1/2$ or $\Delta(\Lambda^-)>1/2$ then $\{\phi(t-\lambda)\}_{\lambda\in\Lambda}$ spans $L^2(\R)$.
\end{enumerate}
Both statements apply also to $L^1(\R)$.
\end{thm}
The proof for $L^2(\R)$ is in section \ref{seccio2} the one for $L^1(\R)$ in section \ref{seccio3}. In section \ref{seccio4} we find generalizations (that improve Zalik's result) and a brief discussion of related problems.

\section{The $L^2(\R)$ case.}\label{seccio2}

The key tool to study the problem is the Bargmann transform:
\begin{equation*}
Bf(z)=2^{\frac{1}{4}}\int_{\R}f(t)e^{2\pi tz-\pi t^2-\frac{\pi}{2}z^2}\,dt,
\end{equation*}
that gives (see \cite{Fol,Charly}) an isomorphism between $L^2(\R)$ and the Fock space:
\begin{equation*}
\mathcal{F}=\left\{f: f \text{ is entire and } \|f\|_{\mathcal{F}}^2=\int_{\C}|f(z)|^2e^{-\pi|z|^2}\, dm(z)<\infty\right\},
\end{equation*}
where $dm$ is the area measure of the plane. This isomorphism allows us to transform our problem into a question of uniqueness sets of the Fock space, so we will be able to use theory of entire functions. Let us precise this statement.
\begin{lem}
Set $\phi(t)=e^{-\pi t^2}$ and let $\Lambda\subset\R$ be a discrete set. Then $\{\phi(t-\lambda)\}_{\lambda\in\Lambda}$ is a complete set for $L^2(\R)$ if and only if $\Lambda$ is an uniqueness set for the Fock space.
\end{lem}
\begin{proof}
We observe that, for $x\in\R$, we have that
\begin{equation*}
Bf(x)=2^{\frac{1}{4}}e^{\frac{\pi}{2}x^2}\langle f,\phi_x\rangle
\end{equation*}
for any $f\in L^2(\R)$. Since any function of the Fock space can be represented in this way in the real line, the values there determine the function and the factor $2^{\frac{1}{4}}e^{\frac{\pi}{2}x^2}\neq 0$, the result follows by duality.
\end{proof}
In fact, the uniqueness sets are sets that can not be included in any zero set of the Fock space. There is no description of the zero sets of the Fock space. We will obtain some properties of these sets using standard theory of entire functions (Hadamard decomposition and Phragm\'{e}n--Lindel\"{o}f principle). Good references for this material are \cite{Lev} or \cite{Boas}.

First one has to observe that functions of the Fock space have order at most $2$, and type at most $\pi/2$.  Also, it is easy to check that any entire function of order less or equal to $2$ and type less than $\pi/2$ is in this space (see \cite{Zhu}). To prove the theorem one uses the properties of the indicator of an entire function $F$ of order $2$:
\begin{equation*}
h_F(\theta)=\limsup_{r\to\infty}\frac{\log|F(re^{i\theta})|}{r^2}.
\end{equation*}
The main properties that we will use are that the indicator (of an order $2$ function) is a $2$-trigonometrically convex function (and therefore continuous) and that the $\sup h_F$ is the type of $F$. We will also use properties of functions of completely regular growth ($\log |F(re^{i\theta})|=h_F(\theta)r^2+o(r^2)$ except for a $C^0$-set). For details and precise statements of all this see \cite{Boas,Lev}.

\begin{proof}[Proof of (1)]
The idea is to construct a function $F$ in the Fock space such that $F(\lambda)=0$ for all $\lambda\in\Lambda$. We suppose for simplicity that $\Delta(\Lambda^+)=\Delta(\Lambda^-)=\Delta$ (if not, one can add points to $\Lambda$, or alternatively make the same arguments but the calculations would not be so easy). We define the set $\Gamma=\Lambda\cup i\Lambda$. This set has angular density
\begin{equation*}
\Delta(\theta)=\lim_{r\to\infty}\frac{|\Gamma\cap\{|z|<r,\arg z<\theta\}|}{r^2}=\Delta\sum_{k=0}^3\chi_{[\frac{k\pi}{2},2\pi)}(\theta).
\end{equation*}
Moreover,
\begin{equation*}
\sum_{\gamma\in\Gamma,|\gamma|<r}\frac{1}{\gamma^2}=0\qquad\forall r.
\end{equation*}
In this conditions one can calculate the indicator function of the Weierstrass canonical product $\Pi_{\Gamma}(z)$ that vanishes in this set:
\begin{equation*}
h_{\Pi}(\theta)=\pi\Delta|\sin 2\theta|,
\end{equation*}
and $\Pi_{\Gamma}$ will be a function of completely regular growth. If $\Delta<1/2$ then $h_{\Pi}<\frac{\pi}{2}$ and $\Pi_{\Gamma}$ is in the Fock space.
\end{proof}

\begin{proof}[Proof of (2)]
We can assume without loss of generality that $\Lambda\subset(0,\infty)$. Let $\Delta=\Delta(\Lambda)$. We define the Weierstrass function $\Pi(z)=\Pi_{\lambda\in\Lambda}(1-z^4/\lambda^4)$. This product defines an entire function ($\Lambda^4$ is a set of order $1/2$) with completely regular growth and its indicator function is $h_{\Pi}(\theta)=\Delta\pi|\sin 2\theta|$. We will compare this function with any other that vanishes in $\Lambda$.

Let $F(z)$ be an entire function of the Fock space such that $F(\Lambda)=0$. As $\Lambda$ has order $2$, Hadamard theorem ensures us that $F$ has order at least $2$. We define:
\begin{equation*}
\varphi(z)=\frac{F(z)}{\Pi(z)}.
\end{equation*}
$\varphi$ is an holomorphic function in any sector of the form $-\pi/2<-\alpha\leqslant\arg z\leqslant\alpha<\pi/2$, and, as $h_{\pi}(\pm\alpha)>0$, in this sector it has order $2$ and mean type (see theorem 5 of section 11.3 in \cite{Lev}). We fix $\alpha>\pi/4$. In the sector $\arg z\in[-\alpha,\alpha]$ we write $F(z)=\varphi(z)\Pi(z)$.

Using that the indicator function of a product of two functions is the sum of their indicator functions if one of them has completely regular growth, we have that:
\begin{equation*}
h_F(\theta)=h_{\varphi}(\theta)+h_{\Pi}(\theta)=h_{\varphi}(\theta)+\pi\Delta|\sin 2\theta|.
\end{equation*}
As $h_{\varphi}$ is a $2$-trigonometrically convex function (because it is the indicator of an order $2$ function), we have that $h_{\varphi}(\theta)+h_{\varphi}(\theta+\pi/2)\geqslant 0$. Then $h_{\varphi}(\pi/4)\geqslant 0$ or $h_{\varphi}(-\pi/4)\geqslant 0$. Thus:
\begin{equation*}
h_{F}(-\frac{\pi}{4})\geqslant \pi\Delta \qquad\text{ or }\qquad h_{F}(\frac{\pi}{4})\geqslant \pi\Delta.
\end{equation*}
If $F$ is in the Fock space we can deduce from the former equation that $\Delta\leqslant1/2$, or equivalently, that $F$ cannot be in the Fock space if $\Delta>1/2$.
\end{proof}

\section{The $L^1(\R)$ case.}\label{seccio3}

By the duality principle,  $\{\phi(t-\lambda)\}_{\lambda\in\Lambda}$ span $L^1(\R)$ if and only if given $f\in L^{\infty}(\R)$, if $\langle f,\phi(t-\lambda)\rangle=0$ for all $\lambda\in\Lambda$ then $f=0$. Then, as $2^{\frac{1}{4}}e^{\frac{\pi}{2}x^2}\langle f,\phi_x\rangle=Bf(x)$, one has to look at the image of $L^{\infty}(\R)$ for the Bargmann transform. However, the Bargmann transform is an isomorphism only in the space $L^2(\R)$. We will obtain the proof of (1) using convolution methods, because one can not ensure that the function $F$ constructed in the former section will belong to the image of $L^{\infty}(\R)$. The proof of (2) will use that $B\bigl(L^{\infty}(\R)\bigr)$ is included in an appropriate space.
\begin{lem}\label{lemma1}
Assume $h\in L^1(\R)$ and that $\widehat{h}(\xi)\neq
0$ for all $\xi$. If $f\in L^p(\R), 1\leq p\leq \infty$ and
the convolution $f*h$ is zero then $f=0$. The same holds if $h\in
L^2(\R)$ and $\widehat{h}(\xi)\neq 0$ almost everywhere.
\end{lem}
\begin{proof}
For $1\leq p\leq 2$, the Fourier transform of $f$ is a function in
$L^{q}, \frac 1p+\frac 1q=1$, and the Fourier transform of $f*h$
is $\widehat{f}\,\widehat{h}$, so the lemma follows. In the
general case, we consider the closed subspace $E$ of $L^1(\R)$
consisting of functions $g$ such that $f*g=0$. Since $E$ is
translation invariant and contains $h$, Wiener's Tauberian theorem
implies that $E$ is the whole $L^1(\R)$, and this implies $f=0$.
When $h\in L^2(\R)$ we use Beurling's theorem describing all
closed translation-invariant subspaces of $L^2(\R)$ to reach the
same result.
\end{proof}
\begin{lem}\label{generaciopertotselsp}
Assume $h\in L^1(\R)\cap L^{\infty}(\R)$ and $\widehat{h}(\xi)\neq 0 $ for
every $\xi$. Then if
$\{\varphi(t-\lambda), \lambda\in\Lambda \}$ span $L^1(\R)$ then
$\{(\varphi\ast h)(t-\lambda), \lambda\in\Lambda\}$ span $L^p(\R)$
for $1\leqslant p<\infty$.
\end{lem}
\begin{proof}
By duality, $\{\varphi(t-\lambda),
\lambda\in\Lambda \}$ span $L^p(\R)$ if and only if
\begin{equation*}
\widetilde{\varphi}*f(\lambda)=\int_{\R}f(t) \varphi(t-\lambda)
\,dt=0 \quad \forall \lambda\in\Lambda
\end{equation*}
implies $f=0$ for $f\in L^q(\R)$. Here $\widetilde{\varphi}(t)=\varphi(-t)$.

For $f\in L^q(\R)$,
\begin{equation*}
\int_{\R} f(t)\, (\varphi\ast h)(t-\lambda)\,
dt=\int_{\R}(\widetilde{h}\ast f)(x)\varphi(x-\lambda)\, dx,
\end{equation*}
whence the result follows from lemma \ref{lemma1} (Hausdorff-Young theorem ensures us that $\varphi\ast h\in L^p(\R)$ and $\widetilde{h}\ast f \in L^{\infty}(\R)$).
\end{proof}
\begin{proof}[Proof of (1) for $L^1(\R)$]
Define $\phi_a(t)=2^{1/4}e^{-a\pi t^2}$ for $a>1$ and observe that if $\frac{1}{a}+\frac{1}{b}=1$ then
\begin{equation*}
\phi_a\ast\phi_b(t)=\frac{(a-1)^{\frac{1}{2}}}{a}\phi(t).
\end{equation*}
If $\{\phi(t-\lambda)\}_{\lambda\in\Lambda}$ span $L^1(\R)$ then $\{\phi_a(t-\sigma)\}_{\sigma\in\frac{1}{\sqrt{a}}\Lambda}$ span it too, hence using lemma \ref{generaciopertotselsp} with $h=\phi_b$ we have that $\{\phi(t-\sigma)\}_{\sigma\in\frac{1}{\sqrt{a}}\Lambda}$ span $L^2(\R)$ for all $a>1$.

As $\Delta(\frac{1}{\sqrt{a}}\Lambda^+)=\frac{1}{a}\Delta(\Lambda^+)$ (respectively for $\Lambda^-$), the statement follows from the $L^2(\R)$ case of theorem \ref{teorema}.
\end{proof}
\begin{lem}
Let $f\in L^{p}(\R)$, $1\leqslant p\leqslant\infty$. The function:
\begin{equation*}
Bf(z)=2^{\frac{1}{4}}\int_{\R}f(t)e^{2\pi tz-\pi t^2-\frac{\pi}{2}z^2}\,dt
\end{equation*}
is an entire function and
\begin{equation*}
|Bf(z)|\leqslant\|f\|_p\|\phi\|_q e^{\frac{\pi}{2}|z|^2},\qquad \frac{1}{p}+\frac{1}{q}=1.
\end{equation*}
\end{lem}
\begin{proof}
We write the Bargmann transform in the following way:
\begin{equation*}
Bf(z)=2^{\frac{1}{4}}e^{\frac{\pi}{2}}\int_{\R}f(t)e^{-\pi(t-z)^2}\,dt.
\end{equation*}
The integral is well defined and defines an entire function for any $f\in L^p(\R)$. As
\begin{equation*}
e^{2\pi t z-\pi t^2-\frac{\pi}{2}z^2}=e^{\frac{\pi}{2}|z|^2}e^{-\pi(t-x)^2}e^{2\pi ity-\pi ixy},
\end{equation*}
we can bound
\begin{equation*}
|Bf(z)|\leqslant e^{\frac{\pi}{2}|z|^2}\int_{\R}|f(t)||e^{-\pi(t-x)^2|}\,dt,
\end{equation*}
thus obtaining the statement by using H\"older inequality.
\end{proof}
\begin{proof}[Proof of (2) for $L^1(\R)$.]
The proof for $L^2(\R)$ applies here without modification because $\Lambda$ is a uniqueness set for entire functions of order $2$ and type $\frac{\pi}{2}$.
\end{proof}
We observe that the theorem can be extended to $L^p(\R)$ for $1<p<2$ and the sufficient condition also for $p>2$. It is an open question whether the discrete sets that can be used to generate $L^p(\R)$ are independent of $p$.

\section{Generalizations and comments.}\label{seccio4}

Using the same ideas we can also prove
\begin{thm}
Fix $n,m\in \N$ and $0<a\leqslant b<\infty$. Let $\phi$ be a function such that
\begin{equation*}
\frac{A}{1+\xi^{2n}}e^{-a\pi\xi^2}\leqslant|\widehat{\varphi}(\xi)|\leqslant B(1+\xi^{2m})e^{-b\pi\xi^2},
\end{equation*}
and $\Lambda$ a discrete set.
\begin{enumerate}
\item If $\Delta(\Lambda^+)<\frac{a}{2}$ and $\Delta(\Lambda^-)<\frac{a}{2}$ then $\{\varphi(t-\lambda)\}_{\lambda\in\Lambda}$ do not span $L^2(\R)$.

\item If $\Delta(\Lambda^+)>\frac{b}{2}$ or $\Delta(\Lambda^-)>\frac{b}{2}$ then $\{\varphi(t-\lambda)\}_{\lambda\in\Lambda}$ span $L^2(\R)$.
\end{enumerate}
This implications also applies to $L^1(\R)$ if we assume $|\widehat{\varphi}'(\xi)|\leqslant C(1+\xi^{2m})e^{-b\pi\xi^2}$.
\end{thm}
To prove this theorem one must use theorem \ref{teorema} and \ref{generaciopertotselsp} with an appropriate $h$. The condition on $\widehat{\varphi}'$ ensures that $\varphi$ and $h$ will be in $L^1(\R)$. This result improves Zalik's one, i. e. the conditions of theorem $2$ of \cite{Zal} are stronger than the ones here.

Zero sets of the Fock space have been studied in \cite{Zhu}. One can also find information about uniqueness sets of zero excess \cite{ALS,LSe99} and the description of sampling \cite{Jan94,Lyu,SW1} and interpolating sets \cite{S1} of this space, with applications to the Gabor transform (details and more information of this transform in \cite{Charly}).

In \cite{Br} and \cite{BOU} one can find the characterization of the discrete sets
$\Lambda\subseteq\R$ for which there exists a function $\varphi\in
L^1(\R)$ with the property that $\{\varphi(t-\lambda)\}_{\lambda\in\Lambda}$ span $L^1(\R)$ as those having infinite
Beurling-Malliavin density. In \cite{Ole} and \cite{OlU} one can
find results proving that in $L^2(\R)$ there are more sets with
this property, and that a characterization in terms of densities
is not possible.

The same problem studied here but for the Poisson function $1/(1+t^2)$ is totally solved in \cite{BrM07}, and generalized in \cite{AB}, where one can found the convolution lemmas stated here. The author does not know any other examples of functions for which such a study has been made. It would be interesting to study this problem for functions such as $e^{-|t|^n}$.

\end{document}